\documentclass[12pt]{elsarticle}
\usepackage{amsfonts}
\usepackage{amsmath}
\usepackage{amssymb}
\usepackage{amsthm}
\usepackage{latexsym}
\usepackage{hyperref}
\usepackage{amsbsy}
\usepackage{bbm}
\usepackage{fullpage}

\usepackage{tikz}
\usetikzlibrary{calc}
\usepackage{mathtools}

\newtheorem{theorem}{Theorem}
\newtheorem{conjec}{Theorem}
\newtheorem{lemma}[theorem]{Lemma}

\newtheorem{proposition}[theorem]{Proposition}

\newtheorem{conjecture}[conjec]{Conjecture}

\journal{Journal of Combinatorial Theory, Series B}

\begin{document}

\begin{frontmatter}
\title{Three conjectures in extremal spectral graph theory}
\author{Michael Tait\corref{cor1}}\ead{mtait@cmu.edu}

\author{Josh Tobin\corref{cor2}}\ead{rjtobin@math.ucsd.edu}

\cortext[cor1]{Some of this research was done while the first author was supported by NSF grant DMS-1606350.} 

\cortext[cor2]{Some of this research was done while both authors were partially supported by NSF grant DMS-1362650.}

\address[add1]{Department of Mathematical Sciences, Carnegie Mellon University}
\address[add2]{Department of Mathematics, University of California, San Diego}

\begin{abstract}We prove three conjectures regarding the maximization of spectral invariants over certain families of graphs. Our most difficult result is that the join of $P_2$ and $P_{n-2}$ is the unique graph of maximum spectral radius over all planar graphs. This was conjectured by Boots and Royle in 1991 and independently by Cao and Vince in 1993. Similarly, we prove a conjecture of Cvetkovi\'c and Rowlinson from 1990 stating that the unique outerplanar graph of maximum spectral radius is the join of a vertex and $P_{n-1}$. Finally, we prove a conjecture of Aouchiche et al from 2008 stating that a pineapple graph is the unique connected graph maximizing the spectral radius minus the average degree. To prove our theorems, we use the leading eigenvector of a purported extremal graph to deduce structural properties about that graph.\end{abstract}

\end{frontmatter}

\section{Introduction}

Questions in {\em extremal graph theory} ask to maximize or minimize a graph invariant over a fixed family of graphs. Perhaps the most well-studied problems in this area are Tur\'an-type problems, which ask to maximize the number of edges in a graph which does not contain fixed forbidden subgraphs. Over a century old, a quintessential example of this kind of result is Mantel's theorem, which states that $K_{\lceil n/2\rceil, \lfloor n/2\rfloor}$ is the unique graph maximizing the number of edges over all triangle-free graphs. {\em Spectral graph theory} seeks to associate a matrix to a graph and determine graph properties by the eigenvalues and eigenvectors of that matrix. This paper studies the maximization of spectral invariants over various families of graphs. We prove three conjectures for $n$ large enough.

\begin{conjecture}[Boots--Royle 1991 \cite{BootsRoyle1991} and independently Cao--Vince 1993 \cite{CaoVince1993}]\label{planar conjecture}
The planar graph on $n\geq 9$ vertices of maximum spectral radius is $P_2 + P_{n-2}$.
\end{conjecture}

\begin{conjecture}[Cvetkovi\'c--Rowlinson 1990 \cite{CvetkovicRowlinson1990}]\label{outerplanar conjecture}
The outerplanar graph on $n$ vertices of maximum spectral radius is $K_1 + P_{n-1}$.
\end{conjecture}

\begin{conjecture}[Aouchiche et al 2008 \cite{AouchicheEtAl2008}]\label{pineapple conjecture}
The connected graph on $n$ vertices that maximizes the spectral radius minus the average degree is a pineapple graph.
\end{conjecture}

In this paper, we prove Conjectures \ref{planar conjecture}, \ref{outerplanar conjecture}, and \ref{pineapple conjecture}, with the caveat that we must assume $n$ is large enough in all of our proofs. We note that the Boots--Royle/Vince--Cao conjecture is not true when $n\in \{7,8\}$ and thus some bound on $n$ is necessary.

For each theorem, the rough structure of our proof is as follows. A lower bound on the invariant of interest is given by the conjectured extremal example. Using this information, we deduce the approximate structure of a (planar, outerplanar, or connected) graph maximizing this invariant. We then use the leading eigenvalue and eigenvector of the adjacency matrix of the graph to deduce structural properties of the extremal graph. Once we know the extremal graph is ``close" to the conjectured graph, we show that it must be exactly the conjectured graph. The majority of the work in each proof is done in the step of using the leading eigenvalue and eigenvector to deduce structural properties of the extremal graph. 

\subsection{History and motivation}

Questions in extremal graph theory ask to maximize or minimize a graph invariant over a fixed family of graphs. This question is deliberately broad, and as such branches into several areas of mathematics. We already mentioned Mantel's Theorem as an example of a theorem in extremal graph theory. Other classic examples include the following. Tur\'an's Theorem \cite{Turan1941} seeks to maximize the number of edges over all $n$-vertex $K_r$-free graphs. The Four Color Theorem seeks to maximize the chromatic number over the family of planar graphs. Questions about maximum cuts over various families of graphs have been studied extensively (cf \cite{Alon1996, BollobasScott2002, Chung1997, GoemansWilliamson1995}). The Erd\H{o}s distinct distance problem seeks to minimize the number of distinct distances between $n$ points in the plane \cite{Erdos1946, GuthKatz2015}.

This paper studies {\em spectral extremal graph theory}, the subset of these extremal problems where invariants are based on the eigenvalues or eigenvectors of a graph. This subset of problems also has a long history of study. Examples include Stanley's bound maximizing spectral radius over the class of graphs on $m$ edges \cite{Stanley1987}, the Alon--Bopanna--Serre Theorem (see \cite{Murty2003, Nilli1991}) and the construction of Ramanujan graphs (see \cite{LubotzkyPhillipsSarnak1988}) minimizing $\lambda_2$ over the family of $d$-regular graphs, theorems of Wilf \cite{Wilf1986} and Hoffman \cite{Hoffman1970} relating eigenvalues of graphs to their chromatic number, and many other examples. Very recently, Bollob\'as, Lee, and Letzter studied maximizing the spectral radius of subgraphs of the hypercube on a fixed number of edges \cite{BollobasLeeLetzter2016}.

A bulk of the recent work in spectral extremal graph theory is by Nikiforov, who has considered maximizing the spectral radius over several families of graphs. Using the fundamental inequality that $\lambda_1(A(G)) \geq 2e(G)/n$, Nikiforov recovers several classic results in extremal graph theory. Among these are spectral strengthenings of Tur\'an's Theorem \cite{Nikiforov2002}, the Erd\H{o}s--Stone--Bollob\'as Theorem \cite{Nikiforov2009ESB}, and the K\H{o}vari--S\'os--Tur\'an Theorem regarding the Zarankiewicz problem \cite{Nikiforov2010} (this was also worked on by Babai and Guiduli \cite{BabaiGuiduli2009}). For many other similar results of Nikiforov, see \cite{Nikiforov2011}. 

We now turn to the history specific to our theorems. 

The study of spectral radius of planar graphs has a long history, dating back to at least Schwenk and Wilson \cite{SchwenkWilson1978}.  This direction
of research was further motivated by applications where the spectral
radius is used as a measure of the connectivity of a network, in
particular for planar
networks in areas such as geography, see for example \cite{BootsRoyle1991}
and its references.  To compare connectivity of networks to a theoretical upper bound, geographers were interested in finding the planar graph of maximum spectral radius. To this end, Boots and Royle and independently Cao and Vince conjectured that the extremal graph is $P_2 + P_{n-2}$
\cite{BootsRoyle1991}, \cite{CaoVince1993}. Several researchers have worked on this problem and successively improved upon the best theoretical upper bound, including \cite{Hong1988},
\cite{CaoVince1993}, \cite{Hong1995}, \cite{Guiduli1996}, \cite{Hong1998},
\cite{EllinghamZha2000}. Other related problems have been considered, for example Dvo\v{r}\'{a}k
and Mohar found an upper bound on the spectral radius of planar graphs
with a given maximum degree \cite{DvorakMohar2010}. Work has also been done maximizing the spectral radius of graphs on surfaces of higher genus \cite{EllinghamZha2000, Hong1995, Hong1998}. We would also like to note that it is claimed in \cite{EllinghamZha2000} that Guiduli and Hayes proved Conjecture \ref{planar conjecture} for sufficiently large $n$. However, this preprint has never appeared, and the authors could not be reached for comment on it.

Conjecture \ref{outerplanar conjecture} appears in \cite{CvetkovicRowlinson1990}, where the authors mention that it is related to the study of various subfamilies of Hamiltonian graphs. Rowlinson \cite{Rowlinson1990} made partial progress on this conjecture, which was also worked on by Cao and Vince \cite{CaoVince1993} and Zhou--Lin--Hu \cite{ZhouLinHu2001}.

Various measures of graph irregularity have been proposed and studied (cf \cite{Albertson1997, Bell1992, CioabaGregory2007, Nikiforov2006} and references therein). These measures capture different aspects of graph irregularity and are incomparable in general. Because of this, a way to understand which graph properties each invariant gauges is to look at the extremal graph. For several of the measures, the graph of maximal irregularity with respect to that measure has been determined \cite{Bell1992, BrightwellWinkler1990, HansenMelot2002, TaitTobin2015}. One such invariant is the spectral radius of the graph minus its average degree, and Conjecture \ref{pineapple conjecture} proposes that the extremal connected graph is a pineapple graph.

\subsection{Notation and preliminaries}

Let $G$ be a connected graph and $A$ the adjacency matrix of $G$. For sets $X,Y\subset V(G)$ we will let $e(X)$ be the number of edges in the subgraph induced by $X$ and $e(X,Y)$ be the number of edges with one endpoint in $X$ and one endpoint in $Y$. For a vertex $v\in V(G)$, we will use $N(v)$ to denote the neighborhood of $v$ and $d_v$ to denote the degree of $v$. For graphs $G$ and $H$, $G+H$ will denote their join.

 Let $\lambda_1 \geq \lambda_2 \geq \cdots \geq \lambda_n$ be the eigenvalues of $A$, and let $\textbf{v}$ be an eigenvector corresponding to $\lambda_1$. By the Perron--Frobenius Theorem, $\textbf{v}$ has all positive entries, and it will be convenient for us to normalize so that the maximum entry is $1$. For a vertex $u\in V(G)$, we will use $\mathbf{v}_u$ to denote the eigenvector entry of $\mathbf{v}$ corresponding to $u$. With this notation, for any $u\in V(G)$, the eigenvector equation becomes
\begin{equation}\label{eigenvector equation}
\lambda_1 \mathbf{v}_u = \sum_{w\sim u} \mathbf{v}_w.
\end{equation}

Throughout the paper, we will use $x$ to denote the vertex with maximum eigenvector entry equal to $1$. If there are multiple such vertices, choose and fix $x$ arbitrarily among them. Since $x=1$, \eqref{eigenvector equation} applied to $x$ becomes
\begin{equation}\label{eigenvector equation for x}
\lambda_1 = \sum_{y\sim x} \mathbf{v}_y.
\end{equation}

Note that this implies $\lambda_1\leq d_x$. The next inequality is a simple consequence of our normalization and an easy double counting argument, but will be used extensively throughout the paper and warrants special attention. We note that this reasoning has been used previously, first by Favaron, Mah\'eo, and Sacl\'e \cite{FavaronMaheoSacle1993}. Multiplying both sides of \eqref{eigenvector equation for x} by $\lambda_1$ and applying \eqref{eigenvector equation} gives
\begin{equation}\label{bound on lambda squared}
\lambda_1^2 = \sum_{y\sim x} \sum_{z\sim y} \mathbf{v}_z = \sum_{y\sim x} \sum_{\substack{z\sim y\\ z\in N(x)}} \mathbf{v}_z + \sum_{y\sim x}\sum_{\substack{z\sim y \\ z\not\in N(x)}} \mathbf{v}_z \leq 2 e\left(N(x)\right) + e\left(N(x), V(G) \setminus N(x), \right)
\end{equation}
where the last inequality follows because each eigenvector entry is at most $1$, and because each eigenvector entry appears at the end of a walk of length $2$ from $x$: each edge with both endpoints in $N(x)$ is the second edge of a walk of length $2$ from $x$ exactly twice and each edge with only one endpoint in $N(x)$ is the second edge of a walk of length $2$ from $x$ exactly once. 

We will also use the Rayleigh quotient characterization of $\lambda_1$:
\begin{equation}\label{rayleigh quotient}
\lambda_1 = \max_{\textbf{z}\not= \textbf{0}} \frac{\textbf{z}^t A \textbf{z}}{\textbf{z}^t \textbf{z}}.
\end{equation}

In particular, this definition of $\lambda_1$ and the Perron--Frobenius Theorem imply that if $H$ is a strict subgraph of $G$, then $\lambda_1(A(G)) > \lambda_1(A(H))$. Another consequence of \eqref{rayleigh quotient} that we use frequently is that $\lambda_1 \geq \frac{2m}{n}$, the average degree of $G$.


\subsection{Applying \eqref{bound on lambda squared}}
Our three main results begin by using \eqref{bound on lambda squared} to deduce structural properties about the corresponding extremal graphs. To illustrate this technique, in this subsection we use \eqref{bound on lambda squared} to give short proofs of two old results. We include this as a quick way for the reader to become aquainted with our notation and how we will use \eqref{bound on lambda squared}. 

\begin{theorem}[Mantel's Theorem]
Let $G$ be a triangle-free graph on $n$ vertices. Then $G$ contains at most $\lfloor n^2/4\rfloor$ edges. Equality occurs if and only if $G = K_{\lfloor n/2 \rfloor \lceil n/2 \rceil}$.
\end{theorem}

\begin{proof}
If $G$ is triangle-free, then $e(N(x)) = 0$. Using $\lambda_1 \geq \frac{2m}{n}$ and \eqref{bound on lambda squared} gives
\[
\frac{4(e(G))^2}{n^2} \leq e(N(x), V(G) \setminus N(x)) \leq \left \lceil \frac{n}{2}\right\rceil \left\lfloor \frac{n}{2}\right\rfloor.
\]

Equality may occur only if $e(N(x), V(G)\setminus N(x)) = \lfloor n^2/4 \rfloor$. The only bipartite graph with this many edges is $ K_{\lfloor n/2 \rfloor \lceil n/2 \rceil}$, and thus $ K_{\lfloor n/2 \rfloor \lceil n/2 \rceil}$ is a subgraph of $G$. But $G$ is triangle-free, and so $G =  K_{\lfloor n/2 \rfloor \lceil n/2 \rceil}$.

\end{proof}

We note that one can attempt to use a similar argument to prove Tur\'an's theorem for $\mathrm{ex}(n, K_r)$, but because of the presence of the term $(e(G))^2$, one must use the integrality of $e(G)$ to deduce the result, and this approach fails when $r$ gets larger than a small constant.

\begin{theorem}[Stanley's Bound \cite{Stanley1987}]
Let $G$ have $m$ edges. Then 
\[
\lambda_1 \leq \frac{1}{2} \left( -1 + \sqrt{1+8m}\right).
\]
Equality occurs if and only if $G$ is a clique and isolated vertices.
\end{theorem}

\begin{proof}
Using \eqref{bound on lambda squared} gives
\[
\lambda_1^2 = \sum_{x\sim y}\sum_{\substack{y\sim z \\ z\not = x}} \mathbf{v}_z +\sum_{x\sim y} 1 \leq 2(m-d_x) + d_x \leq 2m-\lambda_1,
\]
where the last inequality holds because $\lambda_1\leq d_x$. The result follows by the quadratic formula. Examining \eqref{bound on lambda squared} shows that equality holds if and only if $E(G)$ is contained in the closed neighborhood of $x$, $d_x = \lambda_1$, and for each $y\sim x$, $\mathbf{v}_y=1$. Since $x$ was chosen arbitrarily amongst vertices of eigenvector entry $1$, any vertex of eigenvector entry $1$ must contain $E(G)$ in its closed neighborhood. Thus $G$ is a clique plus isolated vertices.
\end{proof}

\subsection{Outline of the paper}
Section \ref{planar} contains our strongest result, the proof of Conjecture \ref{planar conjecture}. In Section \ref{outerplanar} we prove Conjecture \ref{outerplanar conjecture} and in Section \ref{pineapple} we prove Conjecture \ref{pineapple conjecture}.

\section{Outerplanar graphs of maximum spectral radius}\label{outerplanar}
Let $G$ be a graph. As before, let the first eigenvector of the adjacency matrix of $G$ be $\textbf{v}$ normalized so that maximum entry is $1$.  Let $x$ be a vertex with maximum eigenvector entry, i.e. $\mathbf{v}_x=1$.  Throughout let $G$ be an outerplanar graph on $n$ vertices with maximal adjacency spectral radius.  $\lambda_1$ will refer to $\lambda_1(A(G))$. 

Two consequences of $G$ being outerplanar that we will use frequently are that $G$ has at most $2n-3$ edges and $G$ does not contain $K_{2,3}$ as a subgraph. An outline of our proof is as follows. We first show that there is a single vertex of large degree and that the remaining vertices have small eigenvector entry (Lemma \ref{small eigvec}). We use this to show that the vertex of large degree must be adjacent to every other vertex (Lemma \ref{bound bad elements}). From here it is easy to prove that $G$ must be $K_1+P_{n-1}$.

\begin{figure}[]
\begin{center}
\begingroup

\setlength{\unitlength}{.01cm}
{
\setlength{\fboxsep}{10pt}
\framebox[1.5\width]{
\begin{tikzpicture}[rotate=90, scale=1.3]

\coordinate (W) at (1.553380,0.70606);

\coordinate (W1) at ($(W) + 0*(0,-0.7)$);
\coordinate (W2) at ($(W) + 1*(0,-0.7)$);
\coordinate (W3) at ($(W) + 2*(0,-0.7)$);
\coordinate (W4) at ($(W) + 3*(0,-0.7) + (0,-0.25)$);
\coordinate (W5) at ($(W) + 4*(0,-0.7) + (0,-0.25)$);
\coordinate (W6) at ($(W) + 5*(0,-0.7) + (0,-0.25)$);

\coordinate (B) at ($0.5*(W3) + 0.5*(W4) + (-1.0,0)$);

\draw (W1) -- (W2);
\draw (W2) -- (W3);
\draw [dashed] (W3) -- (W4);
\draw (W4) -- (W5);
\draw (W5) -- (W6);

\draw (B) -- (W1);
\draw (B) -- (W2);
\draw (B) -- (W3);
\draw (B) -- (W4);
\draw (B) -- (W5);
\draw (B) -- (W6);

\filldraw[blue] (W) circle (0.07cm);
\filldraw[blue] (W1) circle (0.07cm);
\filldraw[blue] (W2) circle (0.07cm);
\filldraw[blue] (W3) circle (0.07cm);
\filldraw[blue] (W4) circle (0.07cm);
\filldraw[blue] (W5) circle (0.07cm);
\filldraw[blue] (W6) circle (0.07cm);

\filldraw[blue] (B) circle (0.07cm);

\end{tikzpicture}}
}
\endgroup
\end{center}
\caption{The graph $P_1 + P_{n-1}$.
   \label{fig-pn4}}
\end{figure}
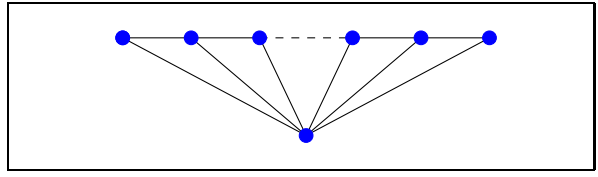

We begin with an easy lemma that is clearly not optimal, but suffices for our needs.
\begin{lemma}\label{trivial lambda bound}
 $\lambda_1 > \sqrt{n-1}$.
\end{lemma}
\begin{proof}
 The star $K_{1,n-1}$ is outerplanar, and cannot be the maximal outerplanar
 graph with respect to spectral radius because it is a strict subgraph of other outerplanar graphs on the same vertex set.  Hence, $\lambda_1(G) > \lambda_1(K_{1,n}) = \sqrt{n-1}$.
\end{proof}

\begin{lemma}
 For any vertex $u$, we have $d_u > \mathbf{v}_un - 11\sqrt{n}$.
\end{lemma}
\begin{proof}
 Let $A$ be the neighborhood of $u$, and let $B = V(G) \setminus (A\cup \{u\})$.  We have 
  \[ \lambda_1^2 \mathbf{v}_u = \sum_{y \sim u} \sum_{z \sim y} \mathbf{v}_z \leq d_u + \sum_{y \sim u} \sum_{z \in N(y)\cap A} \mathbf{v}_z + \sum_{y \sim u} \sum_{z \in N(y)\cap B} z. \]

\noindent By outerplanarity, each vertex in $A$ has at most two neighbors in $A$, otherwise
$G$ would contain a $K_{2,3}$.  In particular,
 \[ \sum_{y \sim u} \sum_{z \in N(y)\cap A} \mathbf{v}_z  \leq 2 \sum_{y \sim u} \mathbf{v}_y = 2\lambda_1 \mathbf{v}_u. \]
Similarly, each vertex in $B$ has at most $2$ neighbors in $A$.  So
 \[ \sum_{y \sim u} \sum_{z \in N(y)\cap B} \mathbf{v}_z \leq 2 \sum_{z \in B} \mathbf{v}_z \leq \frac{2}{\lambda_1} \sum_{z \in B} d_z \leq \frac{4e(G)}{\lambda_1} \leq \frac{4(2n-3)}{\lambda_1}, \]
as  $e(G) \leq 2n-3$ by outerplanarity.  So, using Lemma~\ref{trivial lambda bound} we have
 \[ \sum_{y \sim u} \sum_{z \in N(y)\cap B} \mathbf{v}_z < 8 \sqrt{n}.\]
Combining the above inequalities yields
 \[ \lambda_1^2 \mathbf{v}_u - 2\lambda_1 \mathbf{v}_u < d_u + 8 \sqrt{n}.\]
Again using Lemma~\ref{trivial lambda bound} we get
 \[ \mathbf{v}_u n - 11\sqrt{n} <  (n-1 - 2\sqrt{n-1}) \mathbf{v}_u - 8 \sqrt{n} < d_u .\]
\end{proof}

\begin{lemma}\label{small eigvec}
 We have $d_x > n - 11 \sqrt{n}$ and for every other vertex $u$, $\mathbf{v}_u < C_1 / \sqrt{n}$ for some absolute constant $C_1$, for $n$ sufficiently large.
\end{lemma}
\begin{proof}
The bound on $d_x$ follows immediately from the previous lemma and the normalization that
$\mathbf{v}_x=1$.  Now consider any other vertex $u$.  We know that $G$ contains no $K_{2,3}$, so $d_u < 12 \sqrt{n}$, otherwise $u$ and $x$ share $\sqrt{n}$ neighbors, which
yields a $K_{2,3}$ if $n \geq 9$.  So
 \[ 12 \sqrt{n} > d_u > \mathbf{v}_u n - 11\sqrt{n}, \]
that is, $\mathbf{v}_u < 23 / \sqrt{n}$.
\end{proof}

\begin{lemma}\label{bound bad elements}
 Let $B = V(G) \setminus (N(x) \cup \{x\})$.  Then
  \[ \sum_{z \in B} \mathbf{v}_z < C_2 / \sqrt{n} \]
 for some absolute constant $C_2$.
\end{lemma}
\begin{proof}
From the previous lemma, we have $|B| < 11 \sqrt{n}$.  Now
 \[ \sum_{z \in B} \mathbf{v}_z \leq \frac{1}{\lambda_1} \sum_{z \in B} \left(23 / \sqrt{n}\right) d_z = \frac{23}{\lambda_1 \sqrt{n}} \left( e(A,B) + 2 e(B)\right) . \]
Each vertex in $B$ is adjacent to at most two vertices in $A$, so $e(A,B) \leq 2 |B| < 22 \sqrt{n}$.  The graph induced on $B$ is outerplanar, so
$e(B) \leq 2|B| - 3 < 22 \sqrt{n}$.  Finally, using the fact that $\lambda_1 > \sqrt{n-1}$, we get the required result.
\end{proof}

\begin{theorem}
 For sufficiently large $n$, $G$ is the graph $K_1 + P_{n-1}$, where $+$ represents the graph join operation.
\end{theorem}
\begin{proof}
First we show that the set $B$ above is empty, i.e. $x$ is adjacent
to every other vertex.  If not, let $y \in B$.  Now $y$ is adjacent to at
most two vertices in $A$, and so by Lemma~\ref{small eigvec} and Lemma~\ref{bound bad elements}, 
 \[ \sum_{z \sim y} \mathbf{v}_z < \sum_{z \in B} \mathbf{v}_z + 2 C_1 / \sqrt{n} < (C_2 + 2 C_1) / \sqrt{n} < 1\]
when $n$ is large enough.  Let $G^+$ be the graph obtained
from $G$ by deleting all edges incident to $y$ and replacing them by the single edge $\{x,y\}$.  The resulting graph is outerplanar.  Then,
using the Rayleigh quotient,
 \[ \lambda_1(A^+) - \lambda_1(A) \geq \frac{\textbf{v}^t(A^+ - A)\textbf{v}}{\textbf{v}^t\textbf{v}} = \frac{2\mathbf{v}_y}{\textbf{v}^t\textbf{v}} \left(1 - \sum_{z \sim y} \mathbf{v}_z\right) > 0.\]
This contradicts the maximality of $G$.  Hence $B$ is empty.

Now $x$ is adjacent to every other vertex in $G$.  Hence every vertex other than $x$ has degree less than or equal to $3$.  Moreover, the graph induced by 
$V(G) \setminus \{x\}$ cannot contain any cycles, as then $G$ would not be outerplanar.
It follows that $G$ is a subgraph of $K_1 + P_{n-1}$, and maximality ensures that $G$ must be equal to $K_1 + P_{n-1}$.
\end{proof}

\section{Planar graphs of maximum spectral radius}\label{planar}

As before, let $G$ be a graph with first eigenvector normalized so that maximum entry is $1$, and let $x$ be a vertex with maximum eigenvector entry, i.e. $\mathbf{v}_x=1$. Let $m = |E(G)|$. For subsets $X, Y\subset V(G)$ we write $E(X)$ for the set of edges induced by $X$ and $E(X,Y)$ for the set of edges with one endpoint in $X$ and one endpoint in $Y$.  As before, we let $e(X,Y) = |E(X,Y)|$. We will often assume $n$ is large enough without saying so explicitly. Throughout the section, let $G$ be the planar graph on $n$ vertices with maximum spectral radius, and let $\lambda_1$ denote this spectral radius.

We will use frequently that $G$ has no $K_{3,3}$ as a subgraph, that $m\leq 3n-6$, and that any bipartite subgraph of $G$ has at most $2n-4$ edges. The outline of our proof is as follows. We first show that $G$ has two vertices that are adjacent to most of the rest of the graph (Lemmas \ref{planar trivial spectral bound}--\ref{second vertex of large degree}). We then show that the two vertices of large degree are adjacent (Lemma \ref{x connected to w}), and that they are adjacent to every other vertex (Lemma \ref{A empty}). The proof of the theorem follows readily.

\begin{figure}[]
\begin{center}
\begingroup

\setlength{\unitlength}{.01cm}
{
\setlength{\fboxsep}{10pt}
\framebox[1.5\width]{
\begin{tikzpicture}[rotate=90, scale=1.3]

\coordinate (U) at (0.553380,1.150606);
\coordinate (V) at (2.553380,1.150606);
\coordinate (W) at (1.553380,0.70606);

\coordinate (W1) at ($(W) + 0*(0,-0.7)$);
\coordinate (W2) at ($(W) + 1*(0,-0.7)$);
\coordinate (W3) at ($(W) + 2*(0,-0.7)$);
\coordinate (W4) at ($(W) + 3*(0,-0.7)$);
\coordinate (W5) at ($(W) + 4*(0,-0.7)$);
\coordinate (W6) at ($(W) + 5*(0,-0.9)$);

\draw (U) -- (V);
\draw (U) -- (W);
\draw (V) -- (W);

\draw (W1) -- (W2);
\draw (U) -- (W2);
\draw (V) -- (W2);

\draw (W2) -- (W3);
\draw (U) -- (W3);
\draw (V) -- (W3);

\draw (W3) -- (W4);
\draw (U) -- (W4);
\draw (V) -- (W4);

\draw (W4) -- (W5);
\draw (U) -- (W5);
\draw (V) -- (W5);

\draw [dashed] (W5) -- (W6);
\draw (U) -- (W6);
\draw (V) -- (W6);

\filldraw[blue] (U) circle (0.07cm);
\filldraw[blue] (V) circle (0.07cm);
\filldraw[blue] (W) circle (0.07cm);
\filldraw[blue] (W1) circle (0.07cm);
\filldraw[blue] (W2) circle (0.07cm);
\filldraw[blue] (W3) circle (0.07cm);
\filldraw[blue] (W4) circle (0.07cm);
\filldraw[blue] (W5) circle (0.07cm);
\filldraw[blue] (W6) circle (0.07cm);

\end{tikzpicture}}
}
\endgroup
\end{center}
\caption{The graph $P_2 + P_{n-2}$.
   \label{fig-pn4}}
\end{figure}
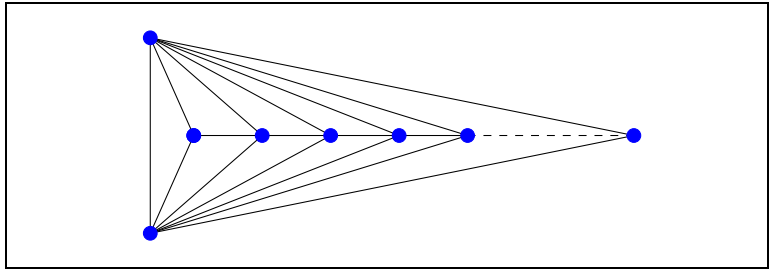

\begin{lemma}\label{planar trivial spectral bound}
$\sqrt{6n} > \lambda_1> \sqrt{2n-4}$.
\end{lemma}
\begin{proof}
For the lower bound, first note that the graph $K_{2,n-2}$ is planar and is a strict subgraph of some other planar graphs on the same vertex set. Since $G$ has maximum spectral radius among all planar graphs on $n$ vertices,
\[
 \lambda_1 > \lambda_1(K_{2,n-2}) = \sqrt{2n-4}.
 \]
 For the upper bound, since the sum of the squares of the eigenvalues equals twice the number of edges in $G$, which is
 at most $6n-12$ by planarity, we get that $\lambda_1 < \sqrt{6n-12} < \sqrt{6n}$.
\end{proof}


Next we partition the graph into vertices of small eigenvector entry and those with large eigenvector entry.  Fix $\epsilon > 0$,
whose exact value will be chosen later.  Let 
\[
L:= \{\mathbf{v}_z\in V(G): \mathbf{v}_z> \epsilon\}
\]
and $S = V(G) \setminus L$. For any vertex $z$, equation~\eqref{eigenvector equation} gives $\mathbf{v}_z\sqrt{2n-4} < \mathbf{v}_z\lambda_1\leq d_z$. Therefore,
\[
2(3n - 6)  \geq \sum_{z\in V(G)} d_z \geq \sum_{z\in L} d_z \geq |L|\epsilon \sqrt{2n-4},
\]
yielding $|L| \leq \frac{3\sqrt{2n-4}}{\epsilon}$. Since the subgraph of $G$ consisting of edges with one endpoint in $L$ and one endpoint in $S$ is a bipartite planar graph, we have $e(S,L) \leq 2n-4$, and since the subgraphs induced by $S$ and by $L$ are each planar, we have $e(S) \leq 3n-6$ and $e(L) \leq \frac{9\sqrt{2n-4}}{\epsilon}$. 


Next we show that there are two vertices adjacent to most of $S$. The first step towards this is an upper bound on the sum of eigenvector entries in both $L$ and $S$.
\begin{lemma}
\begin{equation}\label{eigenvector norm L}
\sum_{z\in L} \mathbf{v}_z \leq  \epsilon \sqrt{2n-4} + \frac{18}{\epsilon}
\end{equation}
\noindent and
 \begin{equation}\label{eigenvector norm S}
 \sum_{z\in S} \mathbf{v}_z \leq (1+3\epsilon)\sqrt{2n-4}.
 \end{equation}

\end{lemma}
\begin{proof}
\[
\sum_{z\in L} \lambda_1 \mathbf{v}_z = \sum_{z\in L} \sum_{y\sim z} \mathbf{v}_y = \sum_{z\in L}\left( \sum_{\substack{y\sim z \\ y\in S}} \mathbf{v}_y + \sum_{\substack{y\sim z\\ y\in L}} \mathbf{v}_y \right) \leq \epsilon e(S,L) + 2e(L) \leq \epsilon (2n-4) + \frac{18\sqrt{2n-4}}{\epsilon}.
\]
Dividing both sides by $\lambda_1$ and using Lemma \ref{planar trivial spectral bound} gives \eqref{eigenvector norm L}.

 On the other hand,
 \[
 \sum_{z\in S} \lambda_1 \mathbf{v}_z = \sum_{z\in S}\sum_{y\sim z} \mathbf{v}_y \leq 2\epsilon e(S) + e(S,L) \leq (6n-12)\epsilon  + (2n-4).
 \]
 Dividing both sides by $\lambda_1$ and using Lemma \ref{planar trivial spectral bound} gives \eqref{eigenvector norm S}.
 \end{proof}

 Now, for $u\in L$ we have
\[
\mathbf{v}_u\sqrt{2n-4} \leq \lambda_1 \mathbf{v}_u=\sum_{y\sim u} \mathbf{v}_y = \sum_{\substack{y\sim u \\ y\in L}} \mathbf{v}_y + \sum_{\substack{y\sim u\\ y\in S}} \mathbf{v}_y \leq \sum_{y\in L} \mathbf{v}_y + \sum_{\substack{y\sim u \\y\in S}} \mathbf{v}_y.
\]
By \eqref{eigenvector norm L}, this gives
\begin{equation}\label{eigenvector norm neighbors of w}
\sum_{\substack{y\sim u \\ y\in S}} \mathbf{v}_y \geq (\mathbf{v}_u-\epsilon)\sqrt{2n-4} - \frac{18}{\epsilon}.
\end{equation}

 The equations \eqref{eigenvector norm S} and \eqref{eigenvector norm neighbors of w} imply that if $u\in L$ and $\mathbf{v}_u$ is close to $1$, then the sum of the eigenvector entries of vertices in $S$ not adjacent to $u$ is small. The following lemma is used to show that $u$ is adjacent to most vertices in $S$.
 
\begin{lemma}\label{eigenvector entry lower bound}
For all $z$ we have $\mathbf{v}_z>\frac{1}{\sqrt{6n}}$.
\end{lemma}
\begin{proof}
  By way of contradiction assume $\mathbf{v}_z\leq \frac{1}{\sqrt{6n}} < \frac{1}{\lambda_1}$. By
  equation~\eqref{eigenvector equation} $z$ cannot be adjacent to $x$, since
  $x$ has eigenvector entry $1$. Let $H$ be the graph obtained from $G$ by removing all edges incident with $z$ and making $z$ adjacent to $x$. Using the Rayleigh quotient, we have $\lambda_1(H) > \lambda_1(G)$, a contradiction.
\end{proof}

\noindent Now letting $u=x$ and combining \eqref{eigenvector norm neighbors of w} and \eqref{eigenvector norm S}, we get 
\[
(1+3\epsilon)\sqrt{2n-4} \geq \sum_{\substack{y\in S\\ y\not\sim x}} \mathbf{v}_y + \sum_{\substack{y\in S\\ y\sim x}} \mathbf{v}_y \geq \sum_{\substack{y\in S\\ y\not\sim x}} \mathbf{v}_y + (1-\epsilon)\sqrt{2n-4} - \frac{18}{\epsilon}.
\]
Now applying Lemma \ref{eigenvector entry lower bound} gives
\[
|\{y\in S: y\not\sim x\}| \frac{1}{\sqrt{6n} }\leq 4\epsilon \sqrt{2n-4} + \frac{18}{\epsilon}.
\]
For $n$ large enough, we have $|\{y\in S: y\not\sim x\}| \leq 14\epsilon n$. So $x$ is adjacent to most of $S$. Our next goal is to show that there is another vertex in $L$ that is adjacent to most of $S$.

\begin{lemma}\label{second vertex of large degree}
There is a $w\in L$ with $w\not=x$ such that $\mathbf{v}_w> 1-24\epsilon$ and $|\{y\in S: y\not\sim w\}| \leq 94\epsilon n$.
\end{lemma}
\begin{proof}
By equation~\eqref{eigenvector equation}, we see
\[
\lambda_1^2 = \sum_{y\sim x}\sum_{z\sim y} \mathbf{v}_z \leq \left(\sum_{uv\in E(G)} \mathbf{v}_u+\mathbf{v}_v\right) - \sum_{y\sim x} \mathbf{v}_y = \left(\sum_{uv\in E(G)} \mathbf{v}_u+\mathbf{v}_v\right) - \lambda_1.
\]
Rearranging and noting that $e(S) \leq 3n-6$ and $e(L) \leq \frac{9\sqrt{2n-4}}{\epsilon}$ since $S$ and $L$ both induce planar subgraphs gives
\begin{align*}
& 2n-4\leq \lambda_1^2 + \lambda_1\leq \sum_{uv\in E(G)} \mathbf{v}_u+\mathbf{v}_v = \left(\sum_{uv\in E(S,L)} \mathbf{v}_u+\mathbf{v}_v\right) + \left(\sum_{uv\in E(S)} \mathbf{v}_u+\mathbf{v}_v\right) + \left(\sum_{uv\in E(L)} \mathbf{v}_u+\mathbf{v}_v \right) \\
& \leq \left(\sum_{uv\in E(S,L)} \mathbf{v}_u+\mathbf{v}_v \right) + \epsilon(6n-12) + \frac{18\sqrt{2n-4}}{\epsilon}.
\end{align*}
So for $n$ large enough,
\[
(2-7\epsilon)n \leq \sum_{uv\in E(S,L)} \mathbf{v}_u+\mathbf{v}_v =\left( \sum_{\substack{uv\in E(S,L)\\ u=x}} \mathbf{v}_u+\mathbf{v}_v\right) + \left(\sum_{\substack{uv\in E(S,L)\\ u\not= x}} \mathbf{v}_u+\mathbf{v}_v\right) \leq \epsilon e(S,L) + d_x + \sum_{\substack{uv\in E(S,L)\\ u\not= x} } \mathbf{v}_u,
\]
giving
\[
\sum_{\substack{uv\in E(S,L)\\ u\not = x}} \mathbf{v}_u \geq (1-9\epsilon)n.
\]

Now since $d_x \geq |S| - 14\epsilon n > (1-15\epsilon)n$, and $e(S,L) < 2n$, the number of terms in the left hand side of the sum is at most $(1+15\epsilon)n$. By averaging, there is a $w\in L$ such that 
\[
\mathbf{v}_w \geq \frac{1-9\epsilon}{1+15\epsilon} > 1-24\epsilon .
\]
Applying \eqref{eigenvector norm neighbors of w} and \eqref{eigenvector norm S} to this $w$ gives
\[
(1+3\epsilon) \sqrt{2n-4} \geq \sum_{\substack{y\in S\\ y\not\sim w}} \mathbf{v}_y + \sum_{\substack{y\in S\\ y\sim w}} \mathbf{v}_y \geq \sum_{\substack{y\in S\\ y\not\sim w}} \mathbf{v}_y + (1-21\epsilon)\sqrt{2n-4} + \frac{18}{\epsilon},
\]
and applying Lemma \ref{eigenvector entry lower bound} gives that for $n$ large enough
\[
|\{y\in S: y\not\sim w\}| \leq 94\epsilon n .
\]

\end{proof}
\medskip

In the rest of the section, let $w$ be the vertex from Lemma \ref{second vertex of large degree}. So $\mathbf{v}_x=1$ and $\mathbf{v}_w> 1-24\epsilon$, and both are adjacent to most of $S$. Our next goal is to show that the remaining vertices are adjacent to both $x$ and $w$. Let $B = N(x) \cap N(w)$ and $A = V(G) \setminus \{x\cup w \cup B\}$. We show that $A$ is empty in two steps: first we show the eigenvector entries of vertices in $A$ are as small as we need, which we then use to show that if there is a vertex in $A$ then $G$ is not extremal.

\begin{lemma}\label{eigenvector entries of A small}
Let $v\in V(G) \setminus \left\{ x,w \right\}$. Then $\mathbf{v}_v < \frac{1}{10}$.
\end{lemma}

\begin{proof}
We first show that the sum over all eigenvector entries in $A$ is small, and then we show that each eigenvector entry is small. Note that for each $v\in A$, $v$ is adjacent to at most one of $x$ and $w$, and is adjacent to at most $2$ vertices in $B$ (otherwise $G$ would contain a $K_{3,3}$ and would not be planar). Thus
\[
\lambda_1\sum_{v\in A} \mathbf{v}_v \leq \sum_{v\in A} d_v \leq 3|A| +2e(A) < 9|A|,
\]
where the last inequality holds by $e(A) < 3|A|$ since $A$ induces a planar graph. Now, since $|L| < \frac{3\sqrt{2n-4}}{\epsilon} < \epsilon n$ for $n$ large enough, we have $|A| \leq (14+94+1)\epsilon n$ (by Lemma \ref{second vertex of large degree}) . Therefore 
\[
\sum_{v\in A} \mathbf{v}_v \leq \frac{9\cdot 109 \cdot \epsilon n}{\sqrt{2n-4}}.
\]

Now any $v\in V(G) \setminus \left\{ x,w \right\}$ is adjacent to at most 4 vertices in $B \cup \left\{x,w \right\}$, as otherwise we would have a $K_{3,3}$
as above.  So we get
\[
\lambda_1\mathbf{v}_v = \sum_{u\sim v} \mathbf{v}_u \leq 4 + \sum_{\substack{u\sim v\\ u\in A}} \mathbf{v}_u \leq 4 + \sum_{u\in A} \mathbf{v}_u \leq C\epsilon\sqrt{n},
\]
where $C$ is an absolute constant not depending on $\epsilon$. Dividing both sides by $\lambda_1$ and choosing $\epsilon$ small enough yields the result.
\end{proof}

We use the fact that the eigenvector entries in $A$ are small to show that if $v\in A$ (i.e. $v$ is not adjacent to both $x$ and $w$), then removing all edges from $v$ and adding edges from it to $x$ and $w$ increases the spectral radius, showing that $A$ must be empty. To do this, we must be able to add edges from a vertex to both $x$ and $w$ and have the resulting graph remain planar. This is accomplished by the following lemma.

\begin{lemma}\label{x connected to w}
If $G$ is extremal, then $x\sim w$.
\end{lemma}

Once $x\sim w$, one may add a new vertex adjacent to only $x$ and $w$ and the resulting graph remains planar.

\begin{proof}[Proof of Lemma \ref{x connected to w}]
From above, we know that for any $\delta > 0$, we may choose $\epsilon$ small enough so that when $n$ is sufficiently
large we have $d_x > (1-\delta)n$ and $d_w > (1-\delta)n$.  By maximality
of $G$, we also know that $G$ has precisely $3n-6$ edges, and by Euler's formula, any planar drawing of $G$ has $2n-4$ faces, each of which is bordered by precisely three edges of $G$ (because in a maximal planar graph, every face is a triangle).

Now we obtain a bound on the number of faces that $x$ and $w$ must be incident to.  Let $X$ be the set of edges incident to $x$.  Each edge in $G$ is incident to precisely two faces, and each face can be incident to at most two edges in $X$ (again, since each face is a triangle by maximality).  So $x$ is incident to at least $|X| = d_x \geq (1-\delta)n$ faces.  Similarly, $w$ is incident to at least $(1-\delta)n$ faces.

Let $F_1$ be the set of faces that are incident to $x$, and then let $F_2$ be the set of faces that are not incident to $x$, but which share an edge with a face in $F_1$.  Let $F = F_1 \cup F_2$.  We have $|F_1| \geq (1-\delta)n$.  Now each
face in $F_1$ shares an edge with exactly three other faces:  if two faces shared two edges, then since each face is a triangle
both faces must be bounded by the same three edges;  this cannot happen, except in the degenerate case when $n=3$.  At most two of these three faces are in $F_1$, and so $|F_2| \geq |F_1| / 3 \geq (1-\delta)n / 3$.  Hence, $|F| \geq (1-\delta)4n / 3$, and so the sum of the number of faces
in $F$ and the number of faces incident to  $w$ is larger than $2n-4$.  In 
particular, there must be some face $f$ that is both belongs to $F$ and is incident to $w$.

Since $f \in F$, then either $f$ is incident to $x$ or $f$ shares an edge with some face that is incident to $x$.  If $f$ is incident to both $x$ and $w$, then $x$ is adjacent to $w$ and 
we are done.  Otherwise, $f$ shares an edge $\left\{y,z\right\}$ with a face $f'$ that is incident to $x$.  In this case, deleting the edge $\left\{y,z\right\}$ and inserting the edge $\left\{x,w\right\}$ yields a planar graph $G'$.  By lemma~\ref{eigenvector entries of A small}, 
the product of the eigenvector entries of $y$ and $z$ is less than $1/100$,
which is smaller than the product of the eigenvector entries of $x$ and $w$.  
This implies that $\lambda_1(G') > \lambda_1(G)$, which is a contradiction.
\end{proof}

We now show that every vertex besides $x$ and $w$ is adjacent to both $x$ and $w$.

\begin{lemma}\label{A empty}
$A$ is empty.
\end{lemma}

\begin{proof}
Assume that $A$ is nonempty. $A$ induces a planar graph, therefore if $A$ is nonempty, then there is a $v\in A$ such that $|N(v)\cap A| < 6$. Further, $v$ has at most $2$ neighbors in $B$ (otherwise $G$ would contain a $K_{3,3}$. Recall that $\textbf{v}$ is the principal eigenvector for the adjacency matrix of $G$. Let $H$ be the graph with vertex set $V(G) \cup \{v'\} \setminus \{v\}$ and edge set $E(H) = E(G\setminus\{v\}) \cup \{v'x, v'w\}$. By Lemma \ref{x connected to w}, $H$ is a planar graph. Then 
\begin{align*}
\textbf{v}^T \textbf{v}\lambda_1(H) &\geq \textbf{v}^T A(H) \textbf{v} & \\
& = \textbf{v}^T A(G) \textbf{v} - 2\sum_{z\sim v} \mathbf{v}_v\mathbf{v}_z + 2\mathbf{v}_v(\mathbf{v}_w+\mathbf{v}_x) & \\
& \geq \textbf{v}^T A(G) \textbf{v} - 14\cdot \mathbf{v}_v \cdot \frac{1}{10} - 2\sum_{\substack{z\sim v\\ z\in \{w,x\}}} \mathbf{v}_v\mathbf{v}_z + 2\mathbf{v}_v(\mathbf{v}_w+\mathbf{v}_x) & \mbox{(by Lemma \ref{eigenvector entries of A small})} \\
& \geq \textbf{v}^T A(G) \textbf{v} - \frac{14}{10}\mathbf{v}_v + 2\mathbf{v}_v\mathbf{v}_w & \mbox{($|N(v)\cap \{x,w\}| \leq 1$)} \\
& > \textbf{v}^T A(G) \textbf{v} & \mbox{(as $\mathbf{v}_w>7/10$)}\\
&= \textbf{v}^T\textbf{v} \lambda_1(G).
\end{align*}
So $\lambda_1(H) > \lambda_1(G)$ and $H$ is planar, i.e. $G$ is not extremal, a contradiction.
\end{proof}

We now have that if $G$ is extremal, then $K_2+I_{n-2}$, the join of an edge and an independent set of size $n-2$, is a subgraph of $G$. Finishing the proof is straightforward.

\begin{theorem}
For $n\geq N_0$, the unique planar graph on $n$ vertices with maximum spectral radius is $K_{2} + P_{n-2}$.
\end{theorem}

\begin{proof}
By Lemmas \ref{x connected to w} and \ref{A empty}, $x$ and $w$ have degree $n-1$. We now look at the set $B = V(G) \setminus \{x,w\}$. For $v\in B$, we have $|N(v) \cap B| \leq 2$, otherwise $G$ contains a copy of $K_{3,3}$. Therefore, the graph induced by $B$ is a disjoint union of paths, cycles, and isolated vertices. However, if there is some cycle $C$ in the graph induced by $B$, then $C \cup\{x,w\}$ is a subdivision of $K_5$. So the graph induced by $B$ is a disjoint union of paths and isolated vertices. However, if $B$ does not induce a path on $n-2$ vertices, then $G$ is a strict subgraph of $K_2 + P_{n-2}$, and we would have $\lambda_1(G) < \lambda_1(K_2 + P_{n-2})$. Since $G$ is extremal, $B$ must induce $P_{n-2}$ and so $G = K_2 + P_{n-2}$.
\end{proof}

\section{Connected graphs of maximum irregularity}\label{pineapple}
Throughout this section, let $G$ be a graph on $n$ vertices with spectral radius $\lambda_1$ and first eigenvector normalized so that $\mathbf{v}_x=1$. Throughout we will use $d = 2e(G)/n$ to denote the average degree. We will also assume that $G$ is the connected graph on $n$ vertices that maximizes $\lambda_1 - d$.

To show that $G$ is a pineapple graph we first show that $\lambda_1 \sim \frac{n}{2}$ and $d\sim \frac{n}{4}$ (Lemma \ref{spectral radius and average degree}). Then we show that there exists a vertex with degree close to $\frac{n}{2}$ and eigenvector entry close to $1$ (Lemma~\ref{u good}). We use this to show that there are many vertices of degree about $\frac{n}{2}$, that these vertices induce a clique, and further that most of the remaining vertices have degree $1$ (Lemma \ref{modifying conditions} and Proposition \ref{almost pineapple structure}).  We complete the proof by showing that all vertices not in the clique have degree $1$ and that they are all adjacent to the same vertex.

We remark that once we show that $G$ is a pineapple graph, the small question remains of {\em which} pineapple graph maximizes $\lambda_1 - d$. Optimization of a cubic polynomial shows that $G$ is a pineapple with clique size $\lceil \frac{n}{2}\rceil +1$ (see \cite{AouchicheEtAl2008}, section 6).

\begin{figure}[]
\begin{center}
\begingroup

\setlength{\unitlength}{.01cm}
{
\setlength{\fboxsep}{10pt}
\framebox[1.5\width]{
\begin{tikzpicture}[rotate=180] 
   
   \coordinate (c1) at ($(7,0) + (180:1)$);
   \coordinate (c2) at ($(7,0) + (225:1)$);
   \coordinate (c3) at ($(7,0) + (270:1)$);
   \coordinate (c4) at ($(7,0) + (315:1)$);
   \coordinate (c5) at ($(7,0) + (0:1)$);
   \coordinate (c6) at ($(7,0) + (45:1)$);
   \coordinate (c7) at ($(7,0) + (90:1)$);
   \coordinate (c8) at ($(7,0) + (135:1)$);
   
   \coordinate (e1) at ($(7,0) + (270:1) + (0.9,-1.0)$);
   \coordinate (e2) at ($(7,0) + (270:1) + (0.5,-1.0)$);
   \coordinate (e3) at ($(7,0) + (270:1) + (-0.5,-1.0)$);
   \coordinate (e4) at ($(7,0) + (270:1) + (-0.9,-1.0)$);
   
   \coordinate (b1) at ($(7,0) + (270:1) + (0.15,-1.0)$);
   \coordinate (b2) at ($(7,0) + (270:1) + (0,-1.0)$);
   \coordinate (b3) at ($(7,0) + (270:1) + (-0.15,-1.0)$);
   
   \coordinate (lb2) at ($(7,0) + (270:1) + (0,-1.4)$);

%
%
%
%
%
%
    
   \draw[draw=black] (c3) -- (e1);   
   \draw[draw=black] (c3) -- (e2);   
   \draw[draw=black] (c3) -- (e3);   
   \draw[draw=black] (c3) -- (e4);   
   
   \draw[black] (7,0) node[circle,minimum size=2cm,draw] (v1) {$K_m$};

   \filldraw[fill=blue,draw=blue] (c3) circle [radius=0.07];

   \filldraw[fill=blue,draw=blue] (e1) circle [radius=0.07];
   \filldraw[fill=blue,draw=blue] (e2) circle [radius=0.07];
   \filldraw[fill=blue,draw=blue] (e3) circle [radius=0.07];
   \filldraw[fill=blue,draw=blue] (e4) circle [radius=0.07];
   
   \filldraw[black] (lb2) node {$n$};   
   
   \filldraw[fill=black] (b1) circle [radius=0.03];
   \filldraw[fill=black] (b2) circle [radius=0.03];   
   \filldraw[fill=black] (b3) circle [radius=0.03];
      
\end{tikzpicture}}
}
\endgroup
\end{center}
\caption{The pineapple graph, $PA(m,n)$.}
   \label{fig-pn4}
\end{figure}
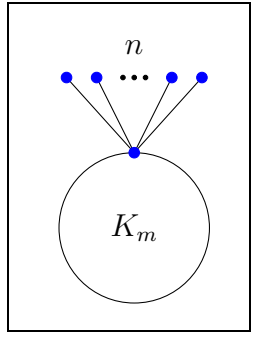

\begin{lemma}\label{spectral radius and average degree}
We have $\lambda_1(G) = \frac{n}{2} + c_1\sqrt{n}$ and $\frac{2e(G)}{n} = \frac{n}{4} + c_2\sqrt{n}$, where $|c_1|, |c_2|<1$.
\end{lemma}

\begin{proof}
By eigenvalue interlacing, $\mathrm{PA}(p,q)$ has spectral radius at least $p-1$. Setting 
 $H = \mathrm{PA}\left( \left\lceil \frac{n}{2}\right\rceil +1, \left\lfloor\frac{n}{2}\right\rfloor-1\right)$, we have 
\[
\lambda_1(H) - \frac{2e(H)}{n} \geq \frac{n}{4} - \frac{3}{2}.
\]
On the other hand, an inequality of Hong \cite{Hong1988} gives 
\[
\lambda_1^2 \leq 2e(G) - (n-1).
\]
It follows that
\begin{equation}\label{lower bound on average degree}
d \geq \frac{\lambda_1^2}{n} + 1 - \frac{1}{n}.
\end{equation}
Setting $\lambda_1 = pn$ and applying \eqref{lower bound on average degree}, we have $\lambda_1 - d \leq pn - p^2n -1 + \frac{1}{n}$. The right hand side of the inequality is maximized at $p=1/2$, giving 
\begin{equation}\label{tight bound on irregularity}
\frac{n}{4} - \frac{3}{2} \leq \lambda_1 - d \leq \frac{n}{4} - 1 + \frac{1}{n}.
\end{equation}

Next setting $\lambda_1 = \frac{n}{2} + c_1 \sqrt{n}$, \eqref{lower bound on average degree} gives
\[
d \geq \frac{n}{4} + c_1\sqrt{n} + c_1^2 + 1 - \frac{1}{n},
\]
whereas \eqref{tight bound on irregularity} implies
\begin{equation}\label{d bar bound}
d \leq \lambda_1 - \frac{n}{4} + \frac{3}{2}  = \frac{n}{4} + c_1\sqrt{n}  + \frac{3}{2}.
\end{equation}
Together, these imply $|c_1| <1$ and prove both statements for $n$ large enough.
\end{proof}

\begin{lemma}\label{error in x neighborhood small}
There exists a constant $c_3$ not depending on $n$ such that 
\[
0 \leq \frac{1}{|N(x)|} \sum_{y\sim x} d_y - \lambda_1 \mathbf{v}_y \leq c_3\sqrt{n}.
\]
\end{lemma}

\begin{proof}
From the inequality of Hong,
\[
\sum_{y\sim x} \lambda_1 \mathbf{v}_y = \lambda_1^2 \leq dn - (n-1).
\]
Rearranging and applying Lemma \ref{spectral radius and average degree}, we have
\[
0\leq \sum_{y\sim x} \left( d_y - \lambda_1 \mathbf{v}_y \right) = O\left(n^{3/2}\right).
\]
By equation~\eqref{eigenvector equation} again, and because the first eigenvector is normalized with $\mathbf{v}_x=1$, we have
\[
\lambda_1 = \sum_{y\sim x} \mathbf{v}_y \leq d_x,
\]
giving $d_x = \Omega(n)$. Combining, we have 
\[
\frac{1}{|N(x)|} \sum_{y\sim x} \left( d_y - \lambda_1 \mathbf{v}_y \right) = O\left(\sqrt{n}\right),
\]
where the implied constant is independent of $n$.
\end{proof}

Now we fix a constant $\epsilon > 0$, whose exact value will be 
chosen later.  The next lemma implies that close to half of the vertices of
$G$ have eigenvector entry close to $1$ for $n$ sufficiently large, 
depending on the chosen $\epsilon$.  We follow that with a proposition
which outlines the approximate structure of $G$, and then finally use
variational arguments to deduce that $G$ is exactly a pineapple graph.

\begin{lemma}\label{u good}
 There exists a vertex $u\not=x$ with $\mathbf{v}_u > 1- 2 \epsilon$ and $d_u - \lambda_1 \mathbf{v}_u = O(\sqrt{n})$.  Moreover $d_u \geq \left( 1/2 - 2\epsilon \right)n$.
\end{lemma}


\begin{proof}
We proceed by first showing a weaker result: that there is a vertex $y$ with
$\mathbf{v}_y> \frac{1}{2} - \epsilon$ and $d_y - \lambda_1 \mathbf{v}_y = O(\sqrt{n})$, and 
additionally that $y \in N(x)$.  We will then use this to obtain
the required result.

Let $A: = \{z\sim x: \mathbf{v}_z> \frac{1}{2} - \epsilon\}$. By Lemma \ref{spectral radius and average degree},
\[
\lambda_1 = \frac{n}{2} + c_1\sqrt{n},
\]
where $|c_1| <1$. Since $0< \mathbf{v}_z\leq 1$ for all $z\sim x$, we see that $|A| \geq \delta_\epsilon n$ where $\delta_\epsilon$ is a positive constant that depends only on $\epsilon$. Let $B = \{z\sim x: d_z - \lambda_1 \mathbf{v}_z >  K\sqrt{n}\}$, where $K$ is a fixed constant whose exact value will be chosen later. Now
\[
\frac{1}{|N(x)|} \sum_{y\sim x} \left( d_y -\lambda_1 \mathbf{v}_y \right) \geq \frac{1}{|N(x)|} \sum_{z\in B} \left( d_z - \lambda_1 \mathbf{v}_z \right) \geq \frac{1}{n} |B| K\sqrt{n}.
\]
By Lemma \ref{error in x neighborhood small}, $|B| \leq \frac{c_3}{K} n$. Therefore, for $K$ large enough depending only on $\epsilon$, we have $\left|A\cap B^c\right| > 0$.  This proves the existence of the vertex $y$,
with the properties claimed at the beginning of the proof.

Next, we show that there exists a set $U\subset N(y)$ such that 
$|U| \geq \left(\frac{1}{4} - 2\epsilon \right)n$ and 
$\mathbf{v}_u \geq 1-2\epsilon$ for all $u\in U$.  By Lemma \ref{spectral radius and average degree},
\[
\left(\frac{n}{2} + c_1\sqrt{n}\right)\left(\frac{1}{2} - \epsilon\right) \leq \lambda_1 \mathbf{v}_y \leq d_y,
\]
where $|c_1| < 1$. So $d_y \geq \left(\frac{1}{4} - \epsilon \right) n$ for $n$ large enough. Now let $C = \{z\sim y: \mathbf{v}_z < 1-2\epsilon \}$. Then
\[
K \sqrt{n} \geq d_y - \lambda_1 \mathbf{v}_y = \sum_{z\sim y} \left( 1 - \mathbf{v}_z \right) \geq \sum_{z\in C} \left( 1-\mathbf{v}_z \right) \geq 2|C| \epsilon.
\]
Therefore
\[
|N(y) \setminus C| \geq \left(\frac{1}{4} - \epsilon\right)n  - \frac{K \sqrt{n}}{2\epsilon}.
\]
Setting $U = N(y) \setminus C$, we have $|U| > \left(\frac{1}{4} - 2\epsilon\right)n$ for $n$ large enough.

Set $D = U \cap N(x)$.  We will first find a lower bound on $|D|$.  We have
 \begin{equation*}
  \lambda_1^2 \leq \sum_{y \sim x} d_y \leq 2m - \sum_{y \not \in N(x)} d_y.
 \end{equation*}
Rearranging this we get 
 \[ d - \frac{\lambda_1^2}{n} \geq \frac{1}{n} \sum_{y \not \in N(x)} d_y.\]
Now applying the bound on $d$ from equation~\ref{d bar bound} and expression for $\lambda_1$ in Lemma~\ref{spectral radius and average degree} yields
 \[ \left( \frac{n}{4} + c_1 \sqrt{n} + \frac{3}{2}\right) - \frac{\left(\frac{n}{2} + c_1 \sqrt{n}\right)^2}{n} \geq \frac{1}{n} \sum_{y \not \in N(x)} d_y, \]
which implies that
 \[ \frac{3}{2} n \geq \left( \frac{3}{2} - c_1^2 \right) n \geq \sum_{y \not \in N(x)} d_y \geq \sum_{y \in U\setminus N(x)} d_y \geq |U\setminus N(x)| (1-2\epsilon) \lambda_1 .\]
So 
 \[ |U \setminus N(x)| \leq \frac{3}{2(1-2\epsilon)} \frac{n}{\lambda_1} = \frac{3}{2(1-2\epsilon)} \frac{1}{1/2 + c_1 n^{-1/2}} .\]
In particular, $|D| \geq (\frac{1}{4} - c'_\epsilon) n$.

Now by the same argument used at the start of the proof to show the existence
of the vertex $y$, we have some vertex $u \in D$
with $d_u - \lambda_1 \mathbf{v}_u = O(\sqrt{n}$).   Finally 
\[d_u \geq \mathbf{v}_u \lambda_1 \geq (1 - 2 \epsilon) (n/2 + c_1\sqrt{n}) \geq \left( 1/2 - 2\epsilon \right)n .  \]

\end{proof}

\begin{lemma}\label{modifying conditions}
 Let $x,y$ be two vertices in $G$.  If
 $\mathbf{v}_x\mathbf{v}_y > 1/2 +n^{-1/2} + 5n^{-1}$, then $x$ and $y$ are adjacent.  On the
 other hand, if $\mathbf{v}_x\mathbf{v}_y < 1/2 - 3\epsilon$ then $x$ and $y$ are not adjacent.
\end{lemma}
\begin{proof}
We begin by bounding the dot product of the leading eigenvector $\textbf{v}$
with itself.  We will show that
\begin{equation}\label{vvt bound}
\frac{n}{2} + \sqrt{n} + 5 \geq \textbf{v}^t\textbf{v} > \frac{n}{2} - 2 \epsilon n  - O(\sqrt{n}).
\end{equation}

\noindent First, we show the lower bound.  With $u$ from the previous lemma, by Cauchy--Schwarz we have
\[  \textbf{v}^t\textbf{v} \geq \sum_{z \sim u} \mathbf{v}_z^2 \geq \frac{1}{d_u}\left( \sum_{z \sim u} \mathbf{v}_z \right)^2 = \frac{(\lambda_1 \mathbf{v}_u)^2}{d_u}. \]
By Lemma~\ref{u good}, we then have
 \[\textbf{v}^t\textbf{v} \geq \frac{(d_u - O(\sqrt{n}))^2}{d_u} \geq d_u - O(\sqrt{n}) > \frac{n}{2} - 2 \epsilon n  - O(\sqrt{n}) .\]
For the upper bound of inequality~\eqref{vvt bound}, first set $E = \left( N(x) \cup \{x\}\right)^C$.  Then
 \[ \textbf{v}^t\textbf{v} = \sum_{z \in V(G)} \mathbf{v}_z^2 \leq \sum_{z \in V(G)} \mathbf{v}_z \leq 1 + \sum_{z \in N(x)} \mathbf{v}_z + \sum_{z \in E} \mathbf{v}_z \leq 1 + \lambda_1 + \frac{1}{\lambda_1} \sum_{z \in E} d_z .\]
From the proof of Lemma~\ref{u good} we have the bound
 \[ \sum_{ z \in E} d_z \leq \frac{3}{2} n .\]
Hence
 \[ \textbf{v}^t \textbf{v}  \leq 1 + \frac{n}{2} + c_1 \sqrt{n} + \frac{3}{2} \cdot \frac{1}{1/2 + c_1 n^{-1/2}} \leq \frac{n}{2} + \sqrt{n} + 5 .\]
This completes the proof of inequality~\eqref{vvt bound}.

 Let $\lambda_1^+$ be the leading eigenvalue of the graph formed by adding the 
 edge $\{x,y\}$ to $G$.  Then by \eqref{rayleigh quotient} we have
  \[ \lambda_1^+ - \lambda_1 \geq \frac{\textbf{v}^t (A^+-A) \textbf{v}}{\textbf{v}^t \textbf{v}} \geq \frac{2\mathbf{v}_x\mathbf{v}_y}{\textbf{v}^t \textbf{v}} \geq \frac{2\mathbf{v}_x\mathbf{v}_y}{n/2 + \sqrt{n} + 5}  = \frac{2\mathbf{v}_x\mathbf{v}_y}{n(1/2+n^{-1/2} + 5 n^{-1})}.\]
If $\mathbf{v}_x\mathbf{v}_y > 1/2+n^{-1/2} + 5n^{-1}$, then
\[ (\lambda_1^+ - d^+) - (\lambda - d) > \frac{2}{n} - \frac{2}{n} = 0.\]
Hence $\left\{x,y\right\}$ must already have been an edge, otherwise this
would contradict the maximality of $G$.

Similarly if $\lambda_1^-$ is the leading eigenvalue of the graph obtained
from $G$ by deleting the edge $\{x,y\}$, then  
  \[ \lambda_1 - \lambda_1^- \leq \frac{\textbf{v}^t (A-A^-) \textbf{v}}{\textbf{v}^t \textbf{v}} \leq \frac{2\mathbf{v}_x\mathbf{v}_y}{n/2 - 2\epsilon n - O(\sqrt{n})} \leq \frac{2\mathbf{v}_x\mathbf{v}_y}{(1/2 - 3 \epsilon)n},\]
when $n$ is large enough.  Now if
$\mathbf{v}_x\mathbf{v}_y < 1/2 - 3\epsilon$, then 
 \[ (\lambda_1 - d) - (\lambda_1^- - d^-) < 0 .\] 

\end{proof}


\begin{figure}[]
\begin{center}
\begingroup

\setlength{\unitlength}{.01cm}
{
\setlength{\fboxsep}{10pt}
\framebox[1.5\width]{
\begin{tikzpicture}
\filldraw[black] (0.553380,1.150606) node[label={[label distance=0.1cm]180:$1 - O(\epsilon)$},label={[label distance=0.1cm]$U$},circle,minimum size=2cm,draw] (v0) {$K_{|U|}$};

\filldraw[black] (4.553380,1.150606) node[label={[label distance=0.1cm]0:$O(\frac{\epsilon}{n})$},label={[label distance=0.1cm]$V$},circle,minimum size=2cm,draw] (v1) {$I_{|V|}$};

\filldraw[black] (2.553380,-1.80606) node[label={[label distance=0.1cm]0:$\frac{1}{2}+O(\epsilon)$},label={[label distance=0.1cm]88:$W$},circle,minimum size=0.8cm,draw] (v1) {$I_{|W|}$};

\coordinate (UC) at (0.553380,1.150606);
\coordinate (VC) at (4.553380,1.150606);
\coordinate (WC) at (2.553380,-1.80606);

\coordinate (Ul1) at ($(UC) + (1.1,-0.4)$);
\coordinate (Ul2) at ($(UC) + (1.18,-0)$);
\coordinate (Ul3) at ($(UC) + (1.1,0.4)$);

\coordinate (Ub2) at ($(UC) + (0.24,-1.2)$);
\coordinate (Ub3) at ($(UC) + (0.54,-1.0)$);

\coordinate (Vl1) at ($(VC) + (-1.1,-0.4)$);
\coordinate (Vl2a) at ($(VC) + (-1.18,-0.2)$);
\coordinate (Vl2b) at ($(VC) + (-1.18,0.2)$);
\coordinate (Vl3) at ($(VC) + (-1.1,0.4)$);

\coordinate (Wt1) at ($(WC) + (-0.5,0.5)$);
\coordinate (Wt2) at ($(WC) + (-0.3,0.6)$);
\coordinate (Wt3) at ($(WC) + (-0.6,0.3)$);

\draw (Ul1) -- (Vl1);
\draw (Ul2) -- (Vl2a);
\draw (Ul2) -- (Vl2b);
\draw (Ul3) -- (Vl3);

\draw (Ub3) -- (Wt1);
\draw (Ub3) -- (Wt2);
\draw (Ub3) -- (Wt3);

\draw (Ub2) -- (Wt3);
\draw (Ub2) -- (Wt1);


\end{tikzpicture}}
}
\endgroup
\end{center}
\caption{Structure of $G$ in Proposition \ref{almost pineapple structure}.  The number beside
each set indicates the values of eigenvector entries within the set. 
$U$ induces a complete
graph and $V$, $W$ are independent sets.  Each vertex in $V$ is adjacent
to exactly one vertex in $U$, and each vertex in $W$ is adjacent to multiple
vertices in $U$.
   \label{pineapple structure picture}}
\end{figure}
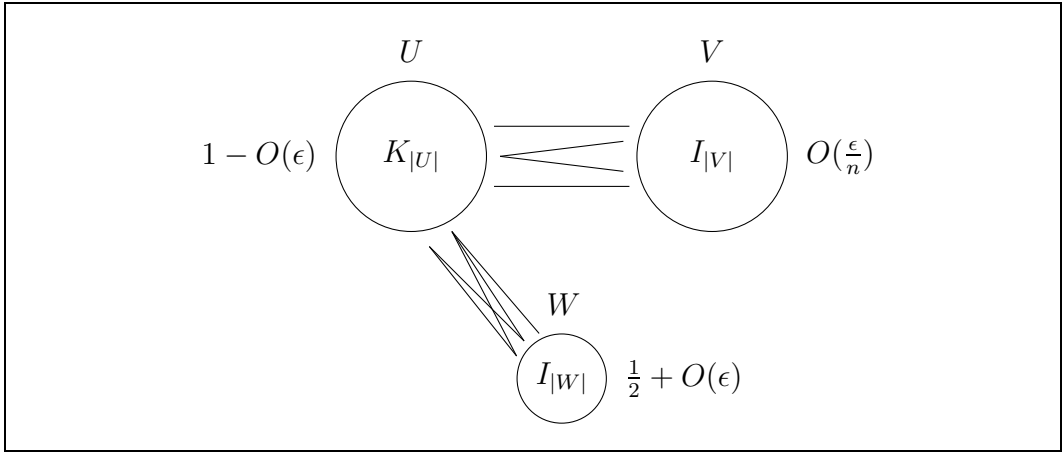


\begin{proposition}\label{almost pineapple structure}
For $n$ sufficiently large, we can partition the vertices of $G$ into three 
sets $U, V, W$ (see Figure \ref{pineapple structure picture}) where 
\begin{itemize}
 \item[(i)] vertices in $V$ have eigenvector entry smaller than $(2+\epsilon) / n$ and have degree one, 
 \item[(ii)] vertices in $U$ induce a clique, 
 all have eigenvector entry larger than $1 - 20\epsilon$, and $(1/2 - 3\epsilon) n \leq |U| \leq (1/2 + \epsilon)n$,

 \item[(iii)] vertices in $W$ have eigenvector entry in the range $\left[1/2 - 4\epsilon,  1/2 + 21 \epsilon \right]$ and are adjacent only to vertices in $U$.  
\end{itemize}
\end{proposition}

\begin{proof}
  By Lemma~\ref{modifying conditions}, any two vertices in $G$ with eigenvector entry $1$ are adjacent.  Moreover, it is easy to see that
  every vertex in $G$ is incident to at least one vertex with eigenvector entry $1$:  if not, for each vertex not incident to a vertex
  with eigenvector entry $1$, delete one of its edges and add a new edge from that vertex to a vertex with eigenvector entry $1$ (such as
  the vertex $x$).  The resulting graph is connected, will have the same number of edges as the original graph, and will have strictly larger $\lambda_1$
  (this can be seen by considering the Rayleigh quotient, as in the proof of Lemma~\ref{modifying conditions}).
  So by maximality of $G$, there are no such vertices.  This implies that the set of edges that are incident
  to a vertex with eigenvector entry $1$ spans the vertex set of $G$. 
  In particular, if we remove any edge that is not incident to a vertex with eigenvector entry $1$, we do not disconnect the graph.  We will use this fact repeatedly in this proof.

\begin{itemize}
\item[(i)] Let $V$ consist of all vertices in $G$ with eigenvector entry 
less than $1/2 - 4 \epsilon$.  By Lemma~\ref{modifying conditions}, removing 
any edge incident to a vertex in $V$ strictly increases $\lambda_1 - d$, so each vertex in $V$ has degree one.  By equation~\eqref{eigenvector equation},
the eigenvector entry of any such vertex is at most $1/\lambda_1 < (2+\epsilon) / n $, when $n$ is large enough.
\item[(ii)] From Lemma~\ref{u good}, we have a vertex $u$ such that $d_u - \lambda_1 \mathbf{v}_u = O(\sqrt{n})$.  Let $X$ be the set of neighbors $z$ of $u$ such that $\mathbf{v}_z < 9/10$.  Then we have
 \[ (1 - 9/10)|X| \leq \sum_{y \sim u} 1 - \mathbf{v}_y = d_u - \lambda_1 \mathbf{v}_u = O(\sqrt{n}). \]
Hence $|X| = O(\sqrt{n})$.  Let $U$ be all vertices in $G$ with eigenvector entry at least $9/10$.  So, by Lemma~\ref{u good}
 \[ |U| \geq d_u - |X| \geq n/2 - 2 \epsilon n - O(\sqrt{n}) . \]
For $n$ large enough, we have $|U| \geq (1/2 - 3\epsilon) n$. For sufficiently large $n$, by Lemma~\ref{modifying conditions} these vertices are all adjacent to each other.  For the upper bound on $|U|$ we use the expression for $e(G)$ in Lemma~\ref{spectral radius and average degree}
\[ |U| (|U| - 1) \leq 2e(G) \leq  \frac{n^2}{4} + c_2n\sqrt{n}, \]
which implies $|U| \leq (1/2+\epsilon) n$ for large enough $n$.

Now take any vertex $y \in U$.  If $x$ is a vertex with largest eigenvector entry, then 
\begin{equation}\label{y_bound}
 \lambda_1 - \lambda_1 \mathbf{v}_y \leq \sum_{z \in N(x) \setminus N(y)} \mathbf{v}_z \leq \mathbf{v}_y + \sum_{z \in U^C} \mathbf{v}_z . 
\end{equation}
We have
 \begin{eqnarray*}
  \lambda_1 \sum_{z \in U^C} \mathbf{v}_z \leq \sum_{z \in U^C} d_z &\leq& 2e(G) - 2|E(U,U)|\\
   &\leq& \frac{n^2}{4} + c_2 n \sqrt{n} - (1/2 - 3\epsilon)(1/2 - 3 \epsilon - 1/n) n^2 \\
   &\leq & 4 \epsilon n^2 ,
 \end{eqnarray*}
for $n$ sufficiently large, where we are using the expression for $e(G)$ given by Lemma~\ref{spectral radius and average degree}.
In particular,
 \[ \sum_{z \in U^C} \mathbf{v}_z \leq 9 \epsilon n .\] 
Finally, by equation~\ref{y_bound} we have 
 \[ \mathbf{v}_y \geq 1 - \frac{1}{\lambda_1} \sum_{z \in U^C} \mathbf{v}_z -\frac{\mathbf{v}_y}{\lambda_1} \geq (1 - 20 \epsilon) .\]
 
\item[(iii)] Let $W$ consist of all remaining vertices of $G$.  If a vertex 
has eigenvector entry smaller than $1/2 - 4\epsilon$ then it is in $V$ by
construction.  If a vertex $z \in W$ has eigenvector entry larger than $1/2 + 21\epsilon$
then we have
 \[ (1/2 + 21 \epsilon) (1 - 20 \epsilon) > 1/2 + \epsilon, \]
if $\epsilon < 1/50$, say.   So for sufficiently large $n$, by Lemma~\ref{modifying conditions} we 
have that $z$ is adjacent to every vertex in $U$.  But by the proof of part (ii), this implies that $\mathbf{v}_z > 1 - 20\epsilon$, which contradicts $z \in W$.

For $z \in W$ and any vertex $y \in U^C$, then 
 $\mathbf{v}_y\mathbf{v}_z \leq (1/2 + 21 \epsilon)(1/2 + 21\epsilon) < 1/4 + 22 \epsilon$
and so by Lemma~\ref{modifying conditions} there is no edge between $y$ and 
$z$ in the maximal graph $G$.
\end{itemize}
\end{proof}

\begin{theorem}
For sufficiently large $n$, $G$ is a pineapple graph.
\end{theorem}
\begin{proof}
  Take $U,V,W$ as in the previous lemma.  We begin by showing that the set $W$ must be empty.
  Proceeding by contradiction, let $z$ be in $W$.  Furthermore
let $G^+$ be the graph obtained by adding edges from $z$ to every vertex in $U$.
We will show that $\lambda_1(G^+) - d(G^+) > \lambda_1(G) - d(G)$, which contradicts the maximality of $G$.

Since the vertex $z$ is adjacent only to vertices in $U$, and the
fact that vertices in $U$ have eigenvector entry between $1-20\epsilon$ and $1$,
equation~\eqref{eigenvector equation} yields
 \[\lambda_1 (1/2 - 4 \epsilon) \leq \lambda_1 \mathbf{v}_z \leq d_z(G) \leq \frac{\lambda_1 \mathbf{v}_z}{1 - 20 \epsilon} = (1/2 + O(\epsilon))\lambda_1 .\]

\noindent Using the expression for $\lambda_1$ in Lemma~\ref{spectral radius and average degree}, for large enough $n$ we have
 \[ \left(1-\epsilon\right)\frac{n}{4} \leq d_z(G) \leq \left(1+\epsilon\right) \frac{n}{4} .\]
So we can bound the change in the average degrees
 \[ d(G^+) - d(G) \leq \frac{2(|U| - (1-\epsilon)n/4)}{n}< 1/2 + 3\epsilon .\]
 Next we find a lower bound on $\lambda_1(G^+) - \lambda_1(G)$.
 Let $\textbf{w}$ be the vector that is equal to $\textbf{v}$ on all 
vertices except $z$, and equal to $1$ for $z$.  Then, 
 \[ \lambda_1(G^+) \geq \frac{\textbf{w}^tA^+\textbf{w}}{\textbf{w}^t\textbf{w}} .\] 
We first find a lower bound for the numerator (with abuse of big-O notation with inequalities)
\begin{eqnarray*} 
\textbf{w}^tA^+\textbf{w} & \geq & \textbf{w}^tA\textbf{w} + 2(|U| - d_z(G))(1-O(\epsilon)) \geq \textbf{w}^tA\textbf{w} + (1/2-O(\epsilon))n \\
& \geq &  \textbf{v}^tA\textbf{v} + 2d_z(G) \left(1-\mathbf{v}_z\right)(1-20\epsilon) + (1/2-O(\epsilon))n \\
& \geq &  \textbf{v}^tA\textbf{v} + 2d_z(G) \left(1/2 - 31\epsilon\right) + (1/2-O(\epsilon))n \\
& \geq &  \textbf{v}^tA\textbf{v} + (3/4-O(\epsilon))n .
\end{eqnarray*}
Similarly, we find an upper bound for the denominator
\begin{eqnarray*}
\textbf{w}^t \textbf{w} &=& \textbf{v}^t\textbf{v} + 1 - \mathbf{v}_z^2 \\
&\leq& \textbf{v}^t\textbf{v} + 1 - (1/2 - 4\epsilon)^2 \\
&\leq& \textbf{v}^t\textbf{v} + 3/4 + 4\epsilon .
\end{eqnarray*}
Combining these, and using the bound on $\textbf{v}^t\textbf{v}$ from
the proof of Lemma~\ref{modifying conditions}, we get
\begin{eqnarray*}
 \lambda_1(G^+) - \lambda_1(G) &\geq& \frac{\textbf{w}^tA^+\textbf{w}}{\textbf{w}^t\textbf{w}} - \frac{\textbf{v}^tA\textbf{v}}{\textbf{v}^t\textbf{v}}\\
 &\geq& \frac{\textbf{v}^t\textbf{v}(3/4-O(\epsilon))n - \textbf{v}^tA\textbf{v} (3/4 + 4\epsilon)}{\textbf{v}^t\textbf{v} (\textbf{v}^t\textbf{v} + 3/4 + 4\epsilon)}\\
 &\geq & \frac{(3/4-O(\epsilon))n - (3/4 + 4\epsilon)\lambda_1(G)}{\textbf{v}^t\textbf{v} + 3/4 + 4\epsilon} \\ 
 &= &  3/4 + O(\epsilon) .
\end{eqnarray*}
Hence $\lambda_1(G^+) - \lambda_1(G) > d(G^+) - d(G)$, and by maximality of $G$ we conclude that $W = \emptyset$.

At this point we know that $G$ consists of a clique together with a set of pendant vertices $V$.  All that remains is to show that all of the pendant vertices  are incident to the same vertex in the clique.  Let $V = \left\{v_1, v_2, \cdots, v_k\right\}$, and let $u_i$ be the unique vertex in $U$ that $v_i$ is adjacent to.  Let $G^+$ be the graph obtained from $G$ by deleting the edges $\left\{v_i,u_i\right\}$ and adding the edges $\left\{v_i,x\right\}$, where $x$ is a vertex with eigenvector entry $1$.  Now, $d(G^+) = d(G)$, and 
\[ \lambda_1(G^+) -\lambda_1(G) \geq \frac{\textbf{v}^t A^+ \textbf{v}}{\textbf{v}^t\textbf{v}} - \frac{\textbf{v}^t A \textbf{v}}{\textbf{v}^t\textbf{v}}, \]
with equality if and only if $\textbf{v}$ is a leading eigenvector for $A^+$.  We have
\[\frac{\textbf{v}^t A^+ \textbf{v}}{\textbf{v}^t\textbf{v}} - \frac{\textbf{v}^t A \textbf{v}}{\textbf{v}^t\textbf{v}} = \frac{1}{\textbf{v}^t\textbf{v}}\left( \sum_{i=1}^k 1 - \mathbf{v}_{u_i} \right) \geq 0 ,\]
with equality if and only if $\mathbf{v}_{u_i} = 1$ for all $1 \leq i \leq k$.  By maximality of $G$, we have equality in both of the above inequalities, and so $\textbf{v}$ is a leading eigenvector for $G^+$, and every vertex in $U$ incident to a vertex in $V$ has eigenvector entry 1.  $G^+$ is a pineapple graph, and it is easy to see that there is a single vertex in a pineapple graph with maximum eigenvector entry.  It follows that the vertices in $V$ are all adjacent to the same vertex in $U$, and hence $G$ is a pineapple graph.

\end{proof}

\section*{Acknowledgements}
We would like to thank Vlado Nikiforov for helpful comments.
\bibliographystyle{plain}
\bibliography{bib}

\begin{thebibliography}{10}

\bibitem{Albertson1997}
Michael~O Albertson.
\newblock The irregularity of a graph.
\newblock {\em Ars Combinatoria}, 46:219--225, 1997.

\bibitem{Alon1996}
Noga Alon.
\newblock Bipartite subgraphs.
\newblock {\em Combinatorica}, 16(3):301--311, 1996.

\bibitem{AouchicheEtAl2008}
Mustapha Aouchiche, Francis~K Bell, Dragi{\v{s}}a Cvetkovi{\'c}, Pierre Hansen,
  Peter Rowlinson, Slobodan~K Simi{\'c}, and Dragan Stevanovi{\'c}.
\newblock Variable neighborhood search for extremal graphs. 16. some
  conjectures related to the largest eigenvalue of a graph.
\newblock {\em European Journal of Operational Research}, 191(3):661--676,
  2008.

\bibitem{BabaiGuiduli2009}
L{\'a}szl{\'o} Babai and Barry Guiduli.
\newblock {Spectral extrema for graphs: the Zarankiewicz problem}.
\newblock {\em The Electronic Journal of Combinatorics}, 16(1):R123, 2009.

\bibitem{Bell1992}
Francis~K Bell.
\newblock A note on the irregularity of graphs.
\newblock {\em Linear Algebra and its Applications}, 161:45--54, 1992.

\bibitem{BollobasLeeLetzter2016}
B{\'e}la Bollob{\'a}s, Jonathan Lee, and Shoham Letzter.
\newblock Eigenvalues of subgraphs of the cube.
\newblock {\em arXiv preprint arXiv:1605.06360}, 2016.

\bibitem{BollobasScott2002}
B\'ela Bollob{\'a}s and Alex~D Scott.
\newblock Better bounds for max cut.
\newblock {\em Contemporary combinatorics}, 10:185--246, 2002.

\bibitem{BootsRoyle1991}
Barry~N Boots and Gordon~F Royle.
\newblock A conjecture on the maximum value of the principal eigenvalue of a
  planar graph.
\newblock {\em Geographical analysis}, 23(3):276--282, 1991.

\bibitem{BrightwellWinkler1990}
Graham Brightwell and Peter Winkler.
\newblock Maximum hitting time for random walks on graphs.
\newblock {\em Random Structures \& Algorithms}, 1(3):263--276, 1990.

\bibitem{CaoVince1993}
Dasong Cao and Andrew Vince.
\newblock The spectral radius of a planar graph.
\newblock {\em Linear Algebra and its Applications}, 187:251--257, 1993.

\bibitem{Chung1997}
Fan~RK Chung.
\newblock {\em Spectral graph theory}, volume~92.
\newblock American Mathematical Society, 1997.

\bibitem{CioabaGregory2007}
Sebastian~M Cioaba and David~A Gregory.
\newblock Principal eigenvectors of irregular graphs.
\newblock {\em Electronic Journal of Linear Algebra}, 16:366--379, 2007.

\bibitem{CvetkovicRowlinson1990}
Dragi{\v{s}}a Cvetkovi{\'c} and Peter Rowlinson.
\newblock The largest eigenvalue of a graph: A survey.
\newblock {\em Linear and multilinear algebra}, 28(1-2):3--33, 1990.

\bibitem{DvorakMohar2010}
Zden{\v{e}}k Dvo{\v{r}}{\'a}k and Bojan Mohar.
\newblock Spectral radius of finite and infinite planar graphs and of graphs of
  bounded genus.
\newblock {\em Journal of Combinatorial Theory, Series B}, 100(6):729--739,
  2010.

\bibitem{EllinghamZha2000}
Mark~N Ellingham and Xiaoya Zha.
\newblock The spectral radius of graphs on surfaces.
\newblock {\em Journal of Combinatorial Theory, Series B}, 78(1):45--56, 2000.

\bibitem{Erdos1946}
Paul Erd{\H{o}}s.
\newblock On sets of distances of $n$ points.
\newblock {\em The American Mathematical Monthly}, 53(5):248--250, 1946.

\bibitem{FavaronMaheoSacle1993}
Odile Favaron, Maryvonne Mah{\'e}o, and J-F Sacl{\'e}.
\newblock {Some eigenvalue properties in graphs (conjectures of Graffiti II)}.
\newblock {\em Discrete Mathematics}, 111(1):197--220, 1993.

\bibitem{GoemansWilliamson1995}
Michel~X Goemans and David~P Williamson.
\newblock Improved approximation algorithms for maximum cut and satisfiability
  problems using semidefinite programming.
\newblock {\em Journal of the ACM (JACM)}, 42(6):1115--1145, 1995.

\bibitem{Guiduli1996}
Barry~D Guiduli.
\newblock {\em Spectral Extrema for Graphs}.
\newblock PhD thesis, University of Chicago, 1996.

\bibitem{GuthKatz2015}
Larry Guth and Nets~Hawk Katz.
\newblock {On the Erd{\H{o}}s distinct distances problem in the plane}.
\newblock {\em Annals of Mathematics}, 181(1):155--190, 2015.

\bibitem{HansenMelot2002}
Pierre Hansen, Hadrien M{\'e}lot, and Groupe d'{\'e}tudes et de recherche en
  analyse~des d{\'e}cisions.
\newblock {\em Variable neighborhood search for extremal graphs 9: Bounding the
  irregularity of a graph}.
\newblock Montr{\'e}al: Groupe d'{\'e}tudes et de recherche en analyse des
  d{\'e}cisions, 2002.

\bibitem{Hoffman1970}
Alan~J Hoffman.
\newblock On eigenvalues and colorings of graphs. 1970 {Graph Theory and its
  Applications} ({Proc. Advanced Sem., Math. Research Center, Univ. of
  Wisconsin, Madison, Wis., 1969}).
\newblock {\em New York}.

\bibitem{LubotzkyPhillipsSarnak1988}
Alexander Lubotzky, Ralph Phillips, and Peter Sarnak.
\newblock Ramanujan graphs.
\newblock {\em Combinatorica}, 8(3):261--277, 1988.

\bibitem{Murty2003}
M~Ram Murty.
\newblock Ramanujan graphs.
\newblock {\em Journal-Ramanujan Mathematical Society}, 18(1):33--52, 2003.

\bibitem{Nikiforov2002}
Vladimir Nikiforov.
\newblock Some inequalities for the largest eigenvalue of a graph.
\newblock {\em Combinatorics, Probability \& Computing}, 11(02):179--189, 2002.

\bibitem{Nikiforov2006}
Vladimir Nikiforov.
\newblock Eigenvalues and degree deviation in graphs.
\newblock {\em Linear Algebra and its Applications}, 414(1):347--360, 2006.

\bibitem{Nikiforov2009ESB}
Vladimir Nikiforov.
\newblock {A spectral Erd{\H{o}}s--Stone--Bollob{\'a}s Theorem}.
\newblock {\em Combinatorics, Probability and Computing}, 18(03):455--458,
  2009.

\bibitem{Nikiforov2010}
Vladimir Nikiforov.
\newblock A contribution to the {Zarankiewicz} problem.
\newblock {\em Linear Algebra and its Applications}, 432(6):1405--1411, 2010.

\bibitem{Nikiforov2011}
Vladimir Nikiforov.
\newblock Some new results in extremal graph theory.
\newblock In {\em Surveys in combinatorics}, pages 213--218. Cambridge
  University Press, 2011.

\bibitem{Nilli1991}
A~Nilli.
\newblock On the second eigenvalue of a graph.
\newblock {\em Discrete Mathematics}, 91(2):207--210, 1991.

\bibitem{Rowlinson1990}
Peter Rowlinson.
\newblock On the index of certain outerplanar graphs.
\newblock {\em Ars Combinatoria}, 29:221--225, 1990.

\bibitem{SchwenkWilson1978}
Allen~J. Schwenk and Robin~J. Wilson.
\newblock On the eigenvalues of a graph.
\newblock In Lowell~W Beineke and Robin~J Wilson, editors, {\em Selected Topics
  in Graph Theory}, chapter~11, pages 307--336. Academic Press, London, 1978.

\bibitem{Stanley1987}
Richard~P Stanley.
\newblock A bound on the spectral radius of graphs with $e$ edges.
\newblock {\em Linear Algebra and its Applications}, 87:267--269, 1987.

\bibitem{TaitTobin2015}
Michael Tait and Josh Tobin.
\newblock Characterizing graphs of maximum principal ratio.
\newblock {\em arXiv preprint arXiv:1511.06378}, 2015.

\bibitem{Turan1941}
Paul Tur{\'a}n.
\newblock On an extremal problem in graph theory.
\newblock {\em Mat. Fiz. Lapok}, 48(137):436--452, 1941.

\bibitem{Wilf1986}
Herbert~S Wilf.
\newblock Spectral bounds for the clique and independence numbers of graphs.
\newblock {\em Journal of Combinatorial Theory, Series B}, 40(1):113--117,
  1986.

\bibitem{Hong1988}
Hong Yuan.
\newblock A bound on the spectral radius of graphs.
\newblock {\em Linear Algebra and its Applications}, 108:135--139, 1988.

\bibitem{Hong1995}
Hong Yuan.
\newblock On the spectral radius and the genus of graphs.
\newblock {\em Journal of Combinatorial Theory, Series B}, 65(2):262--268,
  1995.

\bibitem{Hong1998}
Hong Yuan.
\newblock Upper bounds of the spectral radius of graphs in terms of genus.
\newblock {\em Journal of Combinatorial Theory, Series B}, 74(2):153--159,
  1998.

\bibitem{ZhouLinHu2001}
Jian Zhou, Cuiqin Lin, and Guanzhang Hu.
\newblock Spectral radius of {Hamiltonian} planar graphs and outerplanar
  graphs.
\newblock {\em Tsinghua Science and Technology}, 6(4):350--354, 2001.

\end{thebibliography}

\end{document}